\documentclass[a4paper]{article}

\usepackage{mystyle}

\title{Character varieties for real forms}
\author{Miguel ACOSTA \\
UPMC, 
IMJ-PRG\\
and Université de Lorraine, IECL \\
France
}

\begin{document}
\maketitle

\begin{abstract}
Let $\Gamma$ be a finitely generated group and $G$ a real form of $\slnc$. We propose a definition for the $G$-character variety of $\Gamma$ as a subset of the $\slnc$-character variety of $\Gamma$. We consider two anti-holomorphic involutions of the $\slnc$ character variety and show that an irreducible representation with character fixed by one of them is conjugate to a representation taking values in a real form of $\slnc$. We study in detail an example: the $\slnc$, $\su21$ and $\mathrm{SU}(3)$ character varieties of the free product $\z3z3$.
\end{abstract}

\section{Introduction}

 Character varieties of finitely generated groups have been widely studied and used, whether from the point of view of algebraic geometry or the one of geometric structures and topology. Given a finitely generated group $\Gamma$, and a complex algebraic reductive group $G$, the $G$-character variety of $\Gamma$ is defined as the GIT quotient
 
 \[\mathcal{X}_G(\Gamma) = \mathrm{Hom}(\Gamma, G) // G .\]
It is an algebraic set that takes account of representations of $\Gamma$ with values in $G$ up to conjugacy by an element of $G$.
 See the articles of Sikora \cite{sikora_character_2012} and Heusener \cite{heusener_slnc_2016} for a detailed exposition of the construction.
 Whenever $\Gamma$ has a geometric meaning, for example when it is the fundamental group of a manifold, the character variety reflects its geometric properties. For $\mathrm{SL}_2(\mathbb{C})$-character varieties, we can cite for example the construction of the $A$-polynomial for knot complements, as detailed in the articles \cite{cooper_plane_1994} of Cooper, Culler, Gillet, Long and Shalen, and \cite{cooper_representation_1998} of Cooper and Long, or the considerations related to volume and the number of cusps of a hyperbolic manifold, as well as ideal points of character varieties treated by Morgan and Shalen in \cite{morgan_shalen}, Culler and Shalen in \cite{culler_shalen} and the book of Shalen \cite{shalen_representations_2002}. On the other hand, $\mathrm{SL}_2(\mathbb{C})$-character varieties of compact surface groups are endowed with the Atiyah-Bott-Goldman symplectic structure (see for example \cite{goldman_complex-symplectic_2004}).
  
 In the construction of character varieties, we consider an algebraic quotient $\mathrm{Hom}(\Gamma, G) // G$ where $G$ acts by conjugation. The existence of this quotient as an algebraic set is ensured by the Geometric Invariant Theory (as detailed for example in the article of Sikora \cite{sikora_character_2012}), and it is not well defined for a general algebraic group, nor when considering a non algebraically closed field. Besides that, for the compact form $\mathrm{SU}(n)$,  the classical quotient $\mathrm{Hom}(\Gamma, \mathrm{SU}(n))/\mathrm{SU}(n)$, taken in the sense of topological spaces, is well defined and Hausdorff. See the article \cite{florentino_topology_2009} of Florentino and Lawton for a detailed exposition. This quotient, that Procesi and Schwartz show in their article \cite{procesi_inequalities_1985} to be a semi-algebraic set, can be embedded in the $\slnc$-character variety; we give a proof of this last fact in Section \ref{section_chi_sun}. Similar quotients for other groups have been studied by Parreau in \cite{parreau_espaces_2011}, in which she studies completely reducible representations and in \cite{parreau_compactification_2012}, where she compactifies the space of conjugacy classes of semi-simple representations taking values in noncompact semisimple
connected real Lie groups with finite center.

 It is then natural to try to construct an object similar to a character variety for groups $G$ which are not in the cases stated above, for example real forms of $\slnc$. For the real forms of $\mathrm{SL}_2(\mathbb{C})$, Goldman studies, in his article \cite{goldman_topological_1988}, the real points of the character variety of the rank two free group $F_2$ and shows that they correspond to representations taking values either in $\mathrm{SU}(2)$ or $\mathrm{SL}_2(\mathbb{R})$, which are the real forms of $\mathrm{SL}_2(\mathbb{C})$.
 Inspired by this last approach, we will consider $\slnc$-character varieties and will try to identify the points coming from a representation taking values in a real form of $\slnc$.
 For a finitely generated group $\Gamma$, we introduce two involutions $\Phi_1$ and $\Phi_2$ of the $\slnc$-character variety of $\Gamma$ induced respectively by the involutions $A \mapsto \con{A}$ and $A \mapsto \!^{t}\con{A}^{-1}$ of $\slnc$. We show the following theorem:
 
 \begin{thm}\label{main_thm}
  Let $x$ be a point of the $\slnc$-character variety of $\Gamma$ corresponding to an irreducible representation $\rho$. If $x$ is a fixed point for $\Phi_1$, then $\rho$ is conjugate to a representation taking values in $\slnr$ or $\mathrm{SL}_{n/2}(\mathbb{H})$. If $x$ is a fixed point for $\Phi_2$, then $\rho$ is conjugate to a representation taking values in a unitary group $\mathrm{SU}(p,q)$ with $p+q = n$.
 \end{thm}

  In the second section of this article, we recall the definition of $\slnc$-character varieties with some generalities and examples that will be studied further. In the third section, we recall some generalities on real forms of $\slnc$, we propose a definition for "character varieties for a real form" as a subset of the $\slnc$-character variety and we show Theorem \ref{main_thm} by combining Propositions \ref{prop_invol_traces} and \ref{prop_traces_reelles} in order to identify those character varieties beneath the fixed points of involutions $\Phi_1$ and $\Phi_2$. At last, in Section 4, we study in detail the $\mathrm{SU}(3)$ and $\su21$-character varieties of the free product $\z3z3$. This particular character variety has an interesting geometric meaning since it contains the holonomy representations of two \CR {} uniformizations: the one for the Figure Eight knot complement given by Deraux and Falbel in \cite{falbel} and the one for the Whitehead link complement given by Parker and Will in \cite{ParkerWill}.
  
  \paragraph{Acknowledgements :} The author would like to thank his advisors Antonin Guilloux and Martin Deraux, as well as Pierre Will, Elisha Falbel and Julien Marché for many discussions about this article. He would also like to thank Maxime Wolff, Joan Porti, Michael Heusener, Cyril Demarche and the PhD students of IMJ-PRG for helping him to clarify many points of the paper.
\section{$\mathrm{SL}_n(\mathbb{C})$-character varieties}

 Let $\Gamma$ be a finitely generated group and $n$ a positive integer. We are going to consider representations of $\Gamma$, up to conjugation, taking values in $\slnc$ and in its real forms. In order to study these  representations, the most adapted object to consider is the $\slnc$-character variety. 

\subsection{Definition of character varieties}

 We give here a definition of the $\slnc$-character variety and recall some useful properties.
Character varieties of finitely generated groups have been widely studied, for example in the case of $\mathrm{SL}_2(\mathbb{C})$ by Culler and Shalen in \cite{culler_shalen}. For a detailed exposition of the general results that we state, we refer to the article \cite{sikora_character_2012} of Sikora or the first sections of the article \cite{heusener_slnc_2016} of Heusener.

\begin{defn}
 The $\slnc$-representation variety of $\Gamma$ is the set $\mathrm{Hom}(\Gamma , \mathrm{SL}_n(\mathbb{C}))$. 
\end{defn}

\begin{rem}
 $\mathrm{Hom}(\Gamma , \slnc )$ is an algebraic set, not necessarily irreducible. If $\Gamma$ is generated by  $s_1, \dots , s_k$, an element of $\mathrm{Hom}(\Gamma , \slnc )$ is given by $(S_1, \dots , S_k) \in (\mathcal{M}_n(\mathbb{C}))^k \simeq \mathbb{C}^{kn^2}$ satisfying the equations $\det(S_i) = 1$ for $1 \leq i \leq k$ and, for each relation in the group, the $n^2$ equations in the coefficients associated to an equality  $S_{i_1}^{\alpha_1} \cdots S_{i_l}^{\alpha_l} = \mathrm{Id}$. Since all these equations are polynomial, they define an algebraic set, possibly with many irreducible components. Finally, by changing the generators, we obtain an isomorphism of algebraic sets. 
\end{rem}

\begin{defn}
 The group $\slnc$ acts by conjugation on $\mathrm{Hom}(\Gamma , \slnc )$. The $\slnc$-character variety is the algebraic quotient $\mathrm{Hom}(\Gamma , \slnc ) // \slnc$ by this action. We denote this algebraic set by $\mathcal{X}_{\slnc}(\Gamma)$.
\end{defn}

\begin{rem}\label{git+fonctoriel}
 The existence of this quotient is guaranteed by the Geometric Invariant Theory (GIT), and it is due to the fact that the group $\slnc$ is reductive. By construction, the ring of functions of $\mathcal{X}_{\slnc}(\Gamma)$ is exactly the ring of invariant functions of $\mathrm{Hom}(\Gamma , \slnc )$.
 Moreover, the quotient map is functorial. In particular, if we have a surjective homomorphism $\tilde{\Gamma} \rightarrow \Gamma$, we obtain an injection $\mathrm{Hom}(\Gamma, \slnc) \hookrightarrow \mathrm{Hom}(\tilde{\Gamma}, \slnc)$, which induces an injection $\mathcal{X}_{\slnc}(\Gamma) \hookrightarrow \mathcal{X}_{\slnc}(\tilde{\Gamma})$. See the article \cite{heusener_slnc_2016} of Heusener for a detailed exposition.
 \end{rem}
 


 
The following result, due to Procesi (Theorems 1.3 and 3.3 in \cite{procesi_invariant_1976} ) tells us that it is enough to understand the trace functions in order to understand the whole ring of invariant functions, and that this ring of functions is generated by finitely many trace functions.

\begin{thm}
 The ring of invariant functions of $\mathrm{Hom}(\Gamma , \slnc )$ is generated by the trace functions $\tau_\gamma : \rho \mapsto \mathrm{tr} (\rho(\gamma))$, for $\gamma$ in a finite subset $\{ \gamma_1, \dots , \gamma_k \}$ of $\Gamma$. Consequently, $\mathcal{X}_{\slnc}(\Gamma)$ is isomorphic, as an algebraic set, to the image of $(\tau_{\gamma_1}, \dots ,\tau_{\gamma_k}) : \mathrm{Hom}(\Gamma , \slnc ) \rightarrow \mathbb{C}^k$.
\end{thm}

 Character varieties are strongly related to characters of representations. Let us briefly recall their definition: 

\begin{defn}
 Let $\rho \in \mathrm{Hom}(\Gamma , \slnc )$. The \emph{character of $\rho$} is the function $\chi_\rho : \Gamma \rightarrow \mathbb{C}$ given by $\chi_\rho(g) = \mathrm{tr}(\rho(g))$.
\end{defn}

\begin{rem}
 We have a projection map $\mathrm{Hom}(\Gamma , \slnc ) \rightarrow \mathcal{X}_{\slnc}(\Gamma)$. Two representations  $\rho , \rho' \in \mathrm{Hom}(\Gamma , \slnc )$ have the same image if and only if $\chi_\rho = \chi_{\rho'}$. This explains the name "character variety" for $\mathcal{X}_{\slnc}(\Gamma)$. We will sometimes abusively identify the image of a representation $\rho$ in the character variety with its character $\chi_\rho$. 
\end{rem} 


Semi-simple representations are representations constructed as direct sums of irreducible representations. We will use the following statement when dealing with irreducible representations. 

\begin{thm} \label{thm_semisimple_char} (Theorem 1.28 of \cite{lubotzky_varieties_1985})
 Let $\rho , \rho' \in \mathrm{Hom}(\Gamma , \slnc )$ be two semi-simple representations. Then $\chi_\rho = \chi_{\rho'}$ if and only if $\rho$ and $\rho'$ are conjugate.
\end{thm}

\subsection{Some $\mathrm{SL}_2(\mathbb{C})$ and $\sl3c$-character varieties}
We consider here two  $\sl3c$-character varieties that we will study further: the character variety of the free group of rank two $F_2$ and the one of the fundamental group of the Figure Eight knot complement. We will also recall a classic result describing the $\mathrm{SL}_2(\mathbb{C})$-character variety of $F_2$.

\subsubsection{The free group of rank 2}
We denote here by $s$ and $t$ two generators of the free group of rank two $F_2$, so $F_2 = \langle s,t \rangle$. We will use the character varieties $\mathcal{X}_{\mathrm{SL}_2(\mathbb{C})} (F_2)$ and $\mathcal{X}_{\mathrm{\sl3c}} (F_2)$.
Consider first the following theorem, that describes the $\mathrm{SL}_2(\mathbb{C})$-character variety of $F_2$. A detailed proof can be found in the article of Goldman \cite{goldman_exposition_2004}.

 \begin{thm}[Fricke-Klein-Vogt]
  The character variety $\mathcal{X}_{\mathrm{SL}_2(\mathbb{C})} (F_2)$ is isomorphic to $\mathbb{C}^3$, which is the image of $\mathrm{Hom}(F_2 , \mathrm{SL}_2(\mathbb{C}) )$ by the trace functions of the elements $s,t$ and $st$.
 \end{thm}
 
 \begin{rem}
 Thanks to the theorem below, we know that it is possible to write the trace of the image of $st^{-1}$ in terms of the traces of the images of $s,t$ and $st$ for any representation $\rho : F_2 \rightarrow \mathrm{SL}_2(\mathbb{C})$. By denoting $S$ and $T$ the respective images of $s$ and $t$, the traces of the four elements are related by the \emph{trace equation}:
 \[\mathrm{tr}(S) \mathrm{tr}(T) = \mathrm{tr}(ST) + \mathrm{tr}(ST^{-1}). \]
 \end{rem}

 On the other hand, in his article \cite{lawton_generators_2007}, Lawton describes the $\sl3c$-character variety of $F_2$. He obtains the following result:

\begin{thm}
 $\mathcal{X}_{\mathrm{\sl3c}} (F_2)$ is isomorphic to the algebraic set $V$ of $\mathbb{C}^9$, which is the image of $\mathrm{Hom}(F_2 , \sl3c )$ by the trace functions of the elements $s,t,st,st^{-1}$, of their inverses $s^{-1},t^{-1},t^{-1}s^{-1},ts^{-1}$, and of the commutator $[s,t]$.
 Furthermore, there exist two polynomials $P,Q \in \mathbb{C}[X_1 , \dots , X_8]$ such that $(x_1, \dots , x_9) \in  V$ if and only if $x_9^2 - Q(x_1 , \dots , x_8) x_9 + P(x_1, \dots x_8) = 0$.
\end{thm}

 \begin{rem}
 The polynomials $P$ and $Q$ are explicit: we can find them in the article of Lawton \cite{lawton_generators_2007} or in the survey of Will \cite{will_generateurs}. By denoting $\Delta = Q^2 - 4P$, the algebraic set $V$ is a double cover of $\mathbb{C}^8$, ramified over the zero level set of $\Delta$. Furthermore, the two roots of $X_9^2 - Q(x_1 , \dots , x_8) X_9 + P(x_1, \dots x_8)$, as a polynomial in $X_9$, are given by the traces of the commutators $[s,t]$ and $[t,s] = [s,t]^{-1}$.
 \end{rem}
 
\subsubsection{The Figure Eight knot complement}\label{sous_sect_char_sl3c_m8}
 We state briefly some results about the $\sl3c$-character variety of the figure eight knot complement. It is one of the very few $\sl3c$-character varieties of three-manifolds studied exhaustively. We will come back to it in Subsection \ref{sous_sect_descr_chi_z3z3}.
 The results were obtained independently by Falbel, Guilloux, Koseleff, Rouiller and Thistlethwaite in \cite{character_sl3c}, and by Heusener, Muñoz and Porti in \cite{heusener_sl3c-character_2015}. Denoting by $\Gamma_8$ the fundamental group of the Figure Eight knot complement, they describe the character variety $\mathcal{X}_{\mathrm{\sl3c}} (\Gamma_8)$.
Theorem 1.2 of \cite{heusener_sl3c-character_2015} can be stated in the following way:
 
 \begin{thm}
  The character variety $\mathcal{X}_{\mathrm{\sl3c}} (\Gamma_8)$ has five irreducible components: $X_{\mathrm{TR}}$, $X_{\mathrm{PR}}$, $R_1$, $R_2$, $R_3$. Furthermore:
\begin{enumerate}
 \item The component $X_{\mathrm{TR}}$ only contains characters of completely reducible representations.
 \item The component $X_{\mathrm{PR}}$ only contains characters of reducible representations.
 \item The components $R_1$, $R_2$, $R_3$ contain the characters of irreducible representations.
\end{enumerate}  
 \end{thm}

 \begin{rem}
  We take here the notation $R_1$, $R_2$, $R_3$ given in \cite{character_sl3c}. These components are denoted respectively by $V_0$, $V_1$ and $V_2$ in \cite{heusener_sl3c-character_2015}.
 \end{rem} 

\begin{rem}
\begin{itemize}
\item The component $R_1$ contains the class of the geometric representation, obtained as the holonomy representation $\Gamma_8 \rightarrow \mathrm{PSL}_2(\mathbb{C})$ of the complete hyperbolic structure of the Figure Eight knot followed by the irreducible representation $\mathrm{PSL}_2(\mathbb{C}) \rightarrow \sl3c$.
 
 \item The component $R_3$ is obtained from $R_2$ by a pre-composition with an outer automorphism of $\Gamma_8$. These components contain the representations $\rho_2$ and $\rho_3$ with values in $\su21$ given by Falbel in \cite{falbel_spherical}. The representation $\rho_2$ is the holonomy representation of the \CR {} uniformization of the Figure Eight knot complement given by Deraux and Falbel in \cite{falbel}.
 \end{itemize}
\end{rem}

 Besides determining the irreducible components $R_1$, $R_2$ and $R_3$, Falbel, Guilloux, Koseleff, Rouillier and Thistlethwaite give parameters, in Section 5 of their article \cite{character_sl3c}, for explicit representations corresponding to the points of the character variety.

\section{Character varieties for real forms}
 We are going to be interested in representations of a finitely generated group $\Gamma$ taking values in some real forms of $\slnc$ up to conjugacy. We will focus on the real forms $\mathrm{SU}(3)$ and $\su21$ of $\sl3c$ in the detailed example that we will consider further. In order to study the representations up to conjugacy, we will consider the $\slnc$-character variety and will try to figure out the locus of representations taking values in real forms. When $n=2$, the problem was treated by Morgan and Shalen in \cite{morgan_shalen} and by Goldman in his article \cite{goldman_topological_1988}.

\subsection{Real forms and definition}
 Let us first recall the classification of the real forms of $\slnc$. For a detailed exposition of the results that we state, see the book of Helgason \cite{helgason_geometric_2008}. Recall that a real form of a complex Lie group $G_\mathbb{C}$ is a real Lie group  $G_\mathbb{R}$ such that $G_\mathbb{C} = \mathbb{C} \otimes_\mathbb{R} G_\mathbb{R}$.
 The real forms of $\slnc$ belong to three families: the real groups $\slnr$, the unitary groups $\mathrm{SU}(p,q)$ and the quaternion groups $\mathrm{SL}_{n/2}(\mathbb{H})$. We give the definitions of the two last families in order to fix the notation.

\begin{defn}
Let $n \in \mathbb{N}$ and $p,q \in \mathbb{N}$ such that $n = p+q$. Denote by $I_{p,q}$ the block matrix:
 \[I_{p,q} = \begin{pmatrix}
 I_{p} & 0 \\
 0 & -I_{q}
\end{pmatrix}\] 

We define the group $\mathrm{SU}(p,q)$ as follows:
\[\mathrm{SU}(p,q) = \{M \in \slnc \mid {}^t\!\con{M} I_{p,q} M = I_{p,q} \} .\]
 It is a real Lie group, which is a real form of $\slnc$.
\end{defn}

\begin{defn}
 Let $n \in \mathbb{N}$. Denote by $J_{2n}$ the block matrix:
 \[J_{2n} = \begin{pmatrix}
 0 & I_n \\
 -I_n & 0
\end{pmatrix}\] 

We define the group $\mathrm{SL}_{n}(\mathbb{H})$, also noted $\mathrm{SU}^*(n)$ as follows:
\[\mathrm{SL}_{n}(\mathbb{H}) = \{M \in \mathrm{SL}_{2n}(\mathbb{C}) \mid \con{M}^{-1} J_{2n} M = J_{2n} \} .\]
  It is a real Lie group, which is a real form of $\mathrm{SL}_{2n}(\mathbb{C})$.
\end{defn}

 In order to study representations taking values in real forms, we consider the following definition of character variety for a real form:

\begin{defn}
Let $G$ be a real form of $\slnc$. Let $\Gamma$ be a finitely generated group. We call the $G$-character variety of $\Gamma$ the image of the map $\mathrm{Hom}(\Gamma , G) \rightarrow \mathcal{X}_{\mathrm{SL}_n(\mathbb{C})}(\Gamma)$. In this way,
 
 \[\mathcal{X}_G(\Gamma) = \{ \chi \in \mathcal{X}_{\mathrm{SL}_n(\mathbb{C})} \mid \exists \rho \in  \mathrm{Hom}(\Gamma, G) , \chi = \chi_\rho \}.\]
\end{defn}

 \begin{rem}
 The set $\mathcal{X}_G(\Gamma)$ given by this definition is a subset of a complex algebraic set, which it is not, a priori, a real nor a complex algebraic set. It is the image of a real algebraic set by a polynomial map, and hence a semi-algebraic set. The definition might seem strange if compared to the one for the $\slnc$-character variety. This is due to the fact that the real forms of $\slnc$ are real algebraic groups but not complex algebraic groups and that the algebraic construction and the GIT quotient do not work properly when the field is not algebraically closed. Nevertheless, when considering the compact real form $\mathrm{SU}(n)$, it is possible to define a $\mathrm{SU}(n)$-character variety by considering a topological quotient. We will show, in the next section, that this topological quotient is homeomorphic to the $\mathrm{SU}(n)$-character variety as defined above.
 \end{rem}
 
\subsection{The character variety $\mathcal{X}_{\mathrm{SU}(n)}(\Gamma)$ as a topological quotient}  \label{section_chi_sun}
Let $n$ be a positive integer. We are going to show that the topological quotient $\mathrm{Hom}(\Gamma, \mathrm{SU}(n))/\mathrm{SU}(n)$, where $\mathrm{SU}(n)$ acts by conjugation, is naturally homeomorphic to the character variety $\mathcal{X}_{\mathrm{SU}(n)}(\Gamma)$.
 Let us notice first that a map between these two sets is well defined.
Indeed,
 since two representations taking values in $\mathrm{SU}(n)$ which are conjugate in $\mathrm{SU}(n)$ are also conjugate in $\slnc$, the natural map $\mathrm{Hom}(\Gamma, \mathrm{SU}(n)) \rightarrow \mathcal{X}_{\slnc}(\Gamma)$ factors through the quotient $\mathrm{Hom}(\Gamma, \mathrm{SU}(n)) / \mathrm{SU}(n)$.
 We are going to show the following proposition:

 \begin{prop} \label{prop_chi_sun}
 The map $\mathrm{Hom}(\Gamma,\mathrm{SU}(n)) /  \mathrm{SU}(n) \rightarrow \mathcal{X}_{\mathrm{SU}(n)}(\Gamma)$ is a homeomorphism.
\end{prop}

 \begin{proof}
  We consider $\mathcal{X}_{\mathrm{SU}(n)}(\Gamma)$ as a subset of $\mathcal{X}_{\slnc}(\Gamma) \subset \mathbb{C}^m$, endowed with the usual topology of $\mathbb{C}^m$. By definition, we know that the map $\mathrm{Hom}(\Gamma,\mathrm{SU}(n)) /  \mathrm{SU}(n) \rightarrow \mathcal{X}_{\mathrm{SU}(n)}(\Gamma)$ is continuous and surjective. Since a continuous bijection between a compact space and a Hausdorff space is a homeomorphism, it is enough to show that the map is injective. We want to show that if $\rho_1, \rho_2 \in \mathrm{Hom}(\Gamma, \mathrm{SU}(n))$ are representations such that $\chi_{\rho_1} = \chi_{\rho_2}$, then $\rho_1$ and $\rho_2$ are conjugate in $\mathrm{SU}(n)$. It is the statement of Lemma \ref{lemme_chi_sun}, that we prove below.
 \end{proof}
 
 In order to prove proposition \ref{prop_chi_sun}, we are going to show the following lemma, which seems standard despite the lack of references. 
 
 \begin{lemme}\label{lemme_con_slnc_con_sun}
 Let $\rho_1,\rho_2 \in \mathrm{Hom}(\Gamma, \mathrm{SU}(n))$. If they are conjugate in $\slnc$, then they are conjugate in $\mathrm{SU}(n)$.
\end{lemme}

\begin{proof}
 Let us deal first with the irreducible case, and treat the general case after that.\\
\emph{First case: The representations $\rho_1$ and $\rho_2$ are irreducible.}
 Let $G\in \slnc$ such that $\rho_2 = G\rho_1G^{-1}$.
 Let $J$ be the matrix of the Hermitian form $\sum_{i=1}^{n} \con{x_i}y_i$, which is preserved by the images of $\rho_1$ and $\rho_2$. Since $\rho_2 = G\rho_1G^{-1}$, we know that the image of $\rho_1$ also preserves the form $\!^{t}GJG$. But $\rho_1$ is irreducible: its image preserves then a unique Hermitian form up to a scalar. We deduce that $J = \lambda \!^{t}GJG$ with $\lambda \in \mathbb{R}$.
  Since $J$ is positive definite, we have $\lambda > 0$, and, by replacing $G$ by $\sqrt{\lambda} G$, we have $J = \!^{t}GJG$ i.e. that $G \in \mathrm{SU}(n)$.

\emph{General case.}
 Recall first that every representation $\rho \in \mathrm{Hom}(\Gamma, \mathrm{SU}(n))$ is semi-simple, since its stable subspaces are in an orthogonal sum.

 Let $G \in \slnc$ such that $\rho_2 = G\rho_1G^{-1}$. Since $\rho_1$ takes values in $\mathrm{SU}(n)$, we know that it is semi-simple. It can then be written as a direct sum of irreducible representations: $\rho_1 = \rho_1^{(1)} \oplus \dots \oplus \rho_1^{(m)}$, acting on stable subspaces $E_1, \dots ,E_m$ of $\mathbb{C}^n$, in such a way that the image of $\rho_i$ acts irreducibly on $E_i$ and that $E_i$ are in a direct orthogonal sum. The same applies to $\rho_2$, which admits as stable subspaces $GE_1, \dots , GE_m$. The direct sum $GE_1 \oplus \dots \oplus GE_m$ is then orthogonal. Hence, there exists $U_0 \in \mathrm{SU}(n)$ such that, for all $i \in \{1,\dots , m \}$ we have $U_0G E_i = E_i$. Maybe after conjugating $\rho_2$ by $U_0$, we can suppose that for all $i \in \{1,\dots , m \}$ we have $GE_i = E_i$. The linear map $G$ can then be written $G_1 \oplus \dots \oplus G_m$, acting on $E_1 \oplus \cdots \oplus E_m$. For each $i\in \{1,\dots , m \}$, since $\rho_1^{(i)}$ and $G_i\rho_1^{(i)}G_i^{-1}$ are unitary and irreducible on $E_i$, they are conjugate in $SU(E_i)$ by the first case. We can then replace $G_i$ by $G'_i \in \mathrm{SU}(E_i)$. Setting $G' = G'_1 \oplus \cdots \oplus G'_m$, we have $G' \in \mathrm{SU}(n)$ and $\rho_2 = G'\rho_1G'^{-1}$. 
\end{proof}


 Thanks to Lemma \ref{lemme_con_slnc_con_sun}, we can show the following lemma, which finishes the proof of Proposition \ref{prop_chi_sun} and ensures that $\mathrm{Hom}(\Gamma, \mathrm{SU}(n))/\mathrm{SU}(n) $ and $ \mathcal{X}_{\slnc}(\Gamma)$ are homeomorphic.

\begin{lemme}\label{lemme_chi_sun}
 Let $\rho_1, \rho_2 \in \mathrm{Hom}(\Gamma, \mathrm{SU}(n))$ such that $\chi_{\rho_1} = \chi_{\rho_2}$. Then $\rho_1$ and $\rho_2$ are conjugate in $\mathrm{SU}(n)$.
\end{lemme}

\begin{proof}
 Since $\rho_1$ and $\rho_2$ take values in $\mathrm{SU}(n)$, they are are semi-simple. By Theorem \ref{thm_semisimple_char}, since $\chi_{\rho_1} = \chi_{\rho_2}$ and $\rho_1$ and $\rho_2$ are semi-simple, we know that $\rho_1$ and $\rho_2$ are conjugate in $\slnc$. We deduce, thanks to Lemma \ref{prop_chi_sun}, that $\rho_1$ and $\rho_2$ are conjugate in $\mathrm{SU}(n)$.
\end{proof}

\subsection{ Anti-holomorphic involutions and irreducible representations}
 In this section, we find the locus of character varieties for the real forms of $\slnc$ inside the $\slnc$-character variety $\mathcal{X}_{\slnc}(\Gamma)$. Before focusing on irreducible representations, we show the following proposition, which ensures that two character varieties for two different unitary real forms intersect only in points which correspond to reducible representations. 
 
\begin{prop} \label{prop_inter_charvar_reelle}
 Let $n \in \mathbb{N}$, and $p ,p',q,q' \in \mathbb{N}$ such that $p+q = p'+q' = n$ and $p \neq p', q'$. Let $\rho \in \mathrm{Hom}(\Gamma, \slnc)$ such that $\chi_\rho \in \mathcal{X}_{\mathrm{SU}(p,q)} \cap \mathcal{X}_{\mathrm{SU}(p',q')}$. Then $\rho$ is a reducible representation.
\end{prop} 
\begin{proof}
 Suppose that $\rho$ is irreducible. It is then, up to conjugacy, the only representation of character  $\chi_\rho$. Since $\chi_\rho \in \mathcal{X}_{\mathrm{SU}(p,q)}$, we can suppose that $\rho$ takes values in $\mathrm{SU}(p,q)$. Then, for every $g \in \Gamma$, we have ${}^t\!\con{\rho(g)} J_{p,q} \rho(g) = J_{p,q}$. On the other hand, let us assume that $\rho$ is conjugate to a representation taking values in $\mathrm{SU}(p',q')$. Hence there exists a matrix $J'_{p',q'}$, conjugated to $J_{p',q'}$, such that, for every $g \in \Gamma$, we have ${}^t\!\con{\rho(g)} J'_{p',q'} \rho(g) = J'_{p',q'}$. We deduce that, for every $g \in \Gamma$,
 \[J'_{p',q'} \rho(g)(J'_{p',q'})^{-1} = {}^t\!\con{\rho(g)}^{-1} = J_{p,q} \rho(g)( J_{p,q})^{-1} .\]
The matrix  $(J_{p,q})^{-1} J'_{p',q'}$ commutes with the whole image of $\Gamma$. Since $\rho$ is irreducible, it is a scalar matrix. We deduce that $J_{p,q}$ has either the same signature as $J_{p',q'}$, or the opposite signature. Hence $(p',q')=(p,q)$ or $(p',q')=(q,p)$. 
 
\end{proof}

 From now on, we will limit ourselves to irreducible representations and will consider two anti-holomorphic involutions of the $\slnc$-character variety, which induce anti-holomorphic involutions on the character variety.

We will denote by $\phi_1$ and $\phi_2$ two anti-holomorphic automorphisms of the group $\slnc$, given by $\phi_1(A) = \con{A}$ and $\phi_2(A) = \con{{}^t\!A^{-1}}$. These two involutions induce anti-holomorphic involutions $\Phi_1$ and $\Phi_2$ on the representation variety $\mathrm{Hom}(\Gamma, \slnc)$, in such a way that, for a representation $\rho$, we have $\Phi_1(\rho) = \phi_1 \circ \rho$ and $\Phi_2(\rho) = \phi_2 \circ \rho$.
 
 \begin{prop}
 The involutions $\Phi_1$ and $\Phi_2$ induce as well  anti-holomorphic involutions on the character variety $\mathcal{X}_{\mathrm{\slnc}} (\Gamma)$.
\end{prop} 
 
 \begin{proof}
  For $\rho \in \mathrm{Hom}(\Gamma, \slnc)$ and $g \in \Gamma$ we have $\mathrm{tr}(\Phi_1(\rho)(g)) = \con{ \mathrm{tr}(\rho(g)) }$ and $\mathrm{tr}(\Phi_2(\rho)(g)) = \con{ \mathrm{tr}(\rho(g^{-1})) }$. Hence $\chi_\rho(\Phi_1(g)) = \con{\chi_\rho(g)}$ and $\chi_\rho(\Phi_2(g)) = \con{\chi_\rho(g^{-1})}$. Let $g_1,\dots , g_m \in \Gamma$ such that $\mathcal{X}_{\mathrm{\slnc}} (\Gamma)$ is isomorphic to the image of $\psi : \mathrm{Hom}(\Gamma,\slnc) \rightarrow \mathbb{C}^{2m}$ given by 
  \[\psi(\rho) = ( \chi_\rho (g_1), \dots , \chi_\rho(g_m), \chi_\rho (g_1^{-1}) , \dots , \chi_\rho (g_1^{-1}) ).\]
  
  If $\psi(\rho) = (z_1, \dots , z_m , w_1 , \dots , w_m) \in \mathbb{C}^{2m}$, then 
 $ \psi(\Phi_1(\rho)) = (\con{z_1}, \dots , \con{z_m} , \con{w_1} , \dots , \con{w_m}) $
  and 
 $ \psi(\Phi_2(\rho)) = (\con{w_1} , \dots , \con{w_m} , \con{z_1}, \dots , \con{z_m})$.
 The anti-holomorhic involutions
 \[(z_1, \dots , z_m , w_1 , \dots , w_m) \mapsto (\con{z_1}, \dots , \con{z_m} , \con{w_1} , \dots , \con{w_m})\]
 and
 \[(z_1, \dots , z_m , w_1 , \dots , w_m) \mapsto (\con{w_1} , \dots , \con{w_m} , \con{z_1}, \dots , \con{z_m})\]
 are hence well defined on $\mathcal{X}_{\mathrm{\slnc}} (\Gamma)$ and induced by $\Phi_1$ and $\Phi_2$ respectively.
\end{proof}

We will still denote these involutions on $\mathcal{X}_{\mathrm{\slnc}} (\Gamma)$ by $\Phi_1$ and $\Phi_2$. 
 We will denote by $\mathrm{Fix}(\Phi_1)$ and $\mathrm{Fix}(\Phi_2)$ the points in $\mathcal{X}_{\mathrm{\slnc}} (\Gamma)$ fixed respectively by $\Phi_1$ and $\Phi_2$.

 The following remark ensures that character varieties for real forms are contained in the set of fixed points of  $\Phi_1$ and $\Phi_2$ in $\mathcal{X}_{\slnc}(\Gamma)$: 
 
 \begin{rem}
  If $\rho \in \mathrm{Hom}(\Gamma, \slnc)$ is conjugate to a representation taking values in $\mathrm{SL}_n(\mathbb{R})$, then $\chi_\rho \in \mathrm{Fix}(\Phi_1)$.
 Furthermore, if is conjugate to a representation taking values in $\mathrm{SL}_{n/2}(\mathbb{H})$, then $\chi_\rho \in \mathrm{Fix}(\Phi_1)$, since a matrix $A \in \mathrm{SL}_{n/2}(\mathbb{H})$ is conjugated to $\con{A}$.
   On the other hand, if $\rho$ is conjugate to a representation taking values in $\mathrm{SU}(p,q)$, then $\chi_\rho \in \mathrm{Fix}(\Phi_2)$. Indeed, if $A$ is a unitary matrix, then it is conjugated to $\con{{}^{t}\!A^{-1}}$.
  
  In this way, $\mathcal{X}_{\slnr} (\Gamma) \subset \mathrm{Fix}(\Phi_1)$, $\mathcal{X}_{\mathrm{SL}_{n/2}(\mathbb{H})} (\Gamma) \subset \mathrm{Fix}(\Phi_1)$ and $\mathcal{X}_{\mathrm{SU}(p,q)} (\Gamma) \subset \mathrm{Fix}(\Phi_2)$.
 \end{rem} 
 
 From now on, we will work in the reciprocal direction. We will show that an irreducible representation with character in $\mathrm{Fix}(\Phi_1)$ or $\mathrm{Fix}(\Phi_2)$ is conjugate to a representation taking values in a real form of $\slnc$. Let us begin with the case of $\mathrm{Fix}(\Phi_2)$, which corresponds to unitary groups. The result is given in the following proposition:

\begin{prop}\label{prop_invol_traces}
  Let $\rho \in \mathrm{Hom}(\Gamma, \slnc)$ be an irreducible representation such that $\chi_\rho \in \mathrm{Fix}(\Phi_2)$. Then there exists $p,q \in \mathbb{N}$ with $n=p+q$ such that $\rho$ is conjugate to a representation taking values in $\mathrm{SU}(p,q)$. 
\end{prop}

\begin{proof}
 We know that $\chi_\rho \in \mathrm{Fix}(\Phi_2)$, so the representations $\rho$ and $\Phi_2(\rho)$ have the same character. Since $\rho$ is irreducible, $\rho$ and $\Phi_2(\rho)$ are conjugate. Then there exists $P \in \mathrm{GL}_n(\mathbb{C})$ such that, for every $g \in \Gamma$, we have $P\rho(g)P^{-1} = \con{{}^t\!\rho(g)^{-1}}$. By considering the inverse, conjugating and transposing, we obtain, for every $g \in \Gamma$, that \[\con{{}^t\!P^{-1}}\con{{}^t\!\rho(g)^{-1}}\con{{}^t\!P} = \rho(g).\]
  By replacing $\con{{}^t\!\rho(g)^{-1}}$ in the expression, we deduce that
  \[(P^{-1}\con{{}^t\!P})^{-1} \rho(g) (P^{-1}\con{{}^t\!P}) = \rho(g).\]
  The matrix $P^{-1}\con{{}^t\!P}$ commutes to the whole image of $\rho$. But $\rho$ is irreducible, so there exists $\lambda \in \mathbb{C}$ such that $P^{-1}\con{{}^t\!P} = \lambda \mathrm{Id}$. By taking the determinant, we know that $|\lambda| = 1$. Up to multiplying $P$ by a square root of $\lambda$, we can suppose that $\lambda = 1$. We then have that $P = \con{{}^t\!P} $, which means that $P$ is a Hermitian matrix.  
  We have then a Hermitian matrix $P$ such that, for every $g \in \Gamma$, $\con{{}^t\!\rho(g)} P \rho(g) = P$. The representation $\rho$ takes then values in the unitary group of $P$. 
 Denoting by $(p,q)$ the signature of $P$, the representation $\rho$ is then conjugate to a representation taking values in $\mathrm{SU}(p,q)$.
\end{proof}

\begin{rem}
 \begin{enumerate}
  \item When $n=3$, the only possibilities are $\mathrm{SU}(3)$ and $\su21$.
  \item When $n = 2$, the involutions $\Phi_1$ and $\Phi_2$ are equal: we recognize the result shown by Morgan and Shalen in \cite{morgan_shalen} (Proposition III.1.1) and by Goldman in \cite{goldman_topological_1988} (Theorem 4.3), which is that an irreducible representation with real character is conjugate to either a representation with values in $\mathrm{SU}(2)$, or to a representation with values in $\mathrm{SU}(1,1)$ (appearing as $\mathrm{SL}_2(\mathbb{R})$ for Morgan and Shalen and $\mathrm{SO(2,1)}$ for Goldman).
 \end{enumerate}
\end{rem}

 Let us see now the case of $\mathrm{Fix}(\Phi_1)$, which corresponds to representations taking values in $\mathrm{SL}_n(\mathbb{R})$ or $\mathrm{SL}_{n/2}(\mathbb{H})$. The result is given in the following proposition:

\begin{prop}\label{prop_traces_reelles}
   Let $\rho \in \mathrm{Hom}(\Gamma, \slnc)$ be an irreducible representation such that $\chi_\rho \in \mathrm{Fix}(\Phi_1)$. Then $\rho$ is conjugate to either a representation taking values in $\mathrm{SL}_n(\mathbb{R})$, either a representation taking values in $\mathrm{SL}_{n/2}(\mathbb{H})$ (when $n$ is even).
\end{prop}

We are going to give a proof of this statement inspired from the proof of Proposition \ref{prop_invol_traces}.  An alternative proof can be done by adapting the proof given by Morgan and Shalen in the third part of their article \cite{morgan_shalen} for the $\mathrm{SL}_2(\mathbb{C})$ case.

\begin{lemme}\label{lemme_p_pconj}
 Let $P \in \slnc$ such that $\con{P} P = \mathrm{Id}$. Then, there exists $Q \in \mathrm{GL}_n(\mathbb{C})$ such that $P = \con{Q} Q^{-1}$.
\end{lemme} 
 This fact is an immediate consequence of Hilbert's Theorem 90, which says that $H^1(\mathrm{Gal}(\mathbb{C} / \mathbb{R}) , \slnc)$ is trivial. We give here an elementary proof.
 \begin{proof}
  We search $Q$ of the form $Q_{\alpha} = \alpha \mathrm{Id} + \con{\alpha} \con{P}$. Those matrices satisfy trivially $Q_{\alpha}P = \con{Q_{\alpha}}$. It is then sufficient to find $\alpha \in \mathbb{C}$ such that $\det (Q_{\alpha}) \neq 0$. But $\det (Q_{\alpha}) = \con{\alpha}^n \det (\con{P} + \frac{\alpha}{\con{\alpha}} \mathrm{Id})$, so any $\alpha$ such that $-\frac{\alpha}{\con{\alpha}}$ is not an eigenvalue of $\con{P}$ works.
 \end{proof}
 
 \begin{lemme}\label{lemme_p_pconj_bis}
  Let $P \in \mathrm{SL}_{2m}(\mathbb{C})$ such that $\con{P} P = - \mathrm{Id}$. Then there exists $Q \in \mathrm{GL}_{2m}(\mathbb{C})$ such that $P = \con{Q} J_{2m} Q^{-1}$. 
 \end{lemme}
 \begin{proof}
  We search $Q$ of the form $Q_{\alpha} = -\alpha \mathrm{Id} -  \con{\alpha} J_{2m} \con{P}$. Those matrices satisfy trivially $Q_{\alpha}P = \alpha P + \con{\alpha} J_{2m} = J_{2m} \con{Q_\alpha}$. It is then sufficient to find $\alpha \in \mathbb{C}$ such that $\det (Q_{\alpha}) \neq 0$. But $\det (Q_{\alpha}) = \con{\alpha}^{2n} \det (J_{2m} \con{P} - \frac{\alpha}{\con{\alpha}} \mathrm{Id})$, so any $\alpha$ such that $\frac{\alpha}{\con{\alpha}}$ is not an eigenvalue of $ J_{2m} \con{P}$ works.
 \end{proof}

  \begin{proof}[Proof of Proposition \ref{prop_traces_reelles}]
 We know that $\chi_\rho \in \mathrm{Fix}(\Phi_1)$, so the representations $\rho$ and $\Phi_1(\rho)$ have the same character. Since $\rho$ is irreducible, $\rho$ and $\Phi_1(\rho)$ are conjugate. Hence, there exists $P \in \slnc$ such that, for all $g \in \Gamma$, we have $P\rho(g)P^{-1} = \con{\rho(g)}$. By taking the complex conjugation, we obtain $\con{P}\con{\rho(g)} \con{P^{-1}} = \rho(g)$. By replacing $\con{\rho(g)}$ in the expression, we deduce that for all $g \in \Gamma$:
 \[ (\con{P} P) \rho(g) (\con{P} P)^{-1} = \rho(g).\]
 
  The matrix $\con{P} P$ commutes to the whole image of $\rho$. But $\rho$ is irreducible, so here exists $\lambda \in \mathbb{C}$ such that $\con{P} P = \lambda \mathrm{Id}$. In particular, $P$ and $\con{P}$ commute, so, by conjugating the equality above, we have $\lambda \in \mathbb{R}$. Furthermore, by taking the determinant, we have $\lambda^n = 1$, hence $\lambda = \pm 1$ and $P \con{P} = \pm \mathrm{Id}$. We have two cases:
  
   \emph{First case: $P \con{P} = \mathrm{Id}$. }
 By Lemma \ref{lemme_p_pconj}, there exists $Q \in \slnc$ such that $P = \con{Q}Q^{-1}$.
  
  We deduce that for all $g \in \Gamma$:
   \[\con{Q}Q^{-1} \rho(g) Q \con{Q^{-1}}= \con{\rho(g)}\]
   \[Q^{-1} \rho(g) Q  =  \con{Q^{-1}} \con{\rho(g)} \con{Q} \]
   
This means that the representation $Q^{-1} \rho Q$ takes values in $\mathrm{SL}_n(\mathbb{R})$.

\emph{Second case $P \con{P} = -\mathrm{Id}$.}
By taking the determinant, we see that this case can only happen if $n$ is even. Let $m = \frac{n}{2}$.
 By Lemma \ref{lemme_p_pconj_bis}, there exists $Q \in \mathrm{GL}_{2m}(\mathbb{C})$ such that $P = \con{Q} J_{2m} Q^{-1}$.
We deduce that, for all $g \in \Gamma$:
 \begin{eqnarray*}
 \con{Q} J_{2m} Q^{-1} \rho(g) Q J_{2m}^{-1} \con{Q^{-1}} &=&  \con{\rho(g)}
  \\
 \con{\rho(g)^{-1}} \con{Q} J_{2m} Q^{-1} \rho(g) Q J_{2m} \con{Q^{-1}} &=& \mathrm{Id} 
  \\
 (\con{Q^{-1} \rho(g) Q})^{-1} J_{2m} Q^{-1} \rho(g) Q J_{2m}^{-1} &=&  \mathrm{Id}
 \\
 (\con{Q^{-1} \rho(g) Q})^{-1} J_{2m} Q^{-1} \rho(g) Q  &=& J_{2m}.
 \end{eqnarray*}

   This means that the representation $\con{Q^{-1}} \rho \con{Q}$ takes values in $\mathrm{SL}_m(\mathbb{H})$. 
 \end{proof}
 
 With the propositions below, we showed that an irreducible representation with character in $\mathrm{Fix}(\Phi_1)$ or $\mathrm{Fix}(\Phi_2)$ is conjugate to a representation taking values in a real form of $\slnc$. By combining Propositions \ref{prop_invol_traces} and \ref{prop_traces_reelles} we obtain immediately a proof of Theorem \ref{main_thm}.
 \section{A detailed example: the free product $\z3z3$}

 We are going to study in detail the character varieties $\mathcal{X}_{\su21}(\z3z3)$ and  $\mathcal{X}_{\mathrm{SU}(3)}(\z3z3)$.
 We will begin by studying the character variety $\mathcal{X}_{\sl3c}(\z3z3)$ inside the variety $\mathcal{X}_{\sl3c}(F_2)$ given by Lawton in \cite{lawton_generators_2007}. We will then focus on the fixed points of the involution $\mathrm{Fix}(\Phi_2)$, that will give us the two character varieties with values in real forms, and we will finally describe them in detail and find the slices parametrized by Parker and Will in \cite{ParkerWill} and by Falbel, Guilloux, Koseleff, Rouillier and Thistlethwaite in \cite{character_sl3c}.
 
 \subsection{The character variety $\mathcal{X}_{\sl3c}(\z3z3)$}
 In this section, we will study the character variety $\mathcal{X}_{\sl3c}(\z3z3)$. First, notice that $\z3z3$ is a quotient of the free group of rank two $F_2$. Thanks to remark \ref{git+fonctoriel}, we are going to identify $\mathcal{X}_{\sl3c}(\z3z3)$ as a subset of $\mathcal{X}_{\sl3c}(F_2) \subset \mathbb{C}^9$. Let us begin by making some elementary remarks on order 3 elements of $\sl3c$.
 
 \begin{rem}
 \begin{itemize}
 \item If $S \in \sl3c$, then the characteristic polynomial of $S$ is $\chi_S = X^3 - \mathrm{tr}(S)X^2 + \mathrm{tr}(S^{-1})X -1$.
  \item If $S \in \sl3c$ is of order 3, then $S^3 - \mathrm{Id} = 0$. Hence the matrix $S$ is diagonalizable and admits as eigenvalues cube roots of 1. We will denote this cube roots by $1$, $\omega$ and $\omega^2$.
  \end{itemize}
 \end{rem} 
 
 The following elementary lemma will be useful to separate the irreducible components of $\mathcal{X}_{\sl3c}(\z3z3)$.
 
 \begin{lemme}\label{lemme_ordre3_sl3c}
  Let $S \in \sl3c$. The following assertions are equivalent:
   \begin{enumerate}
    \item $S^3 = \mathrm{Id}$
    \item One of the following cases holds:
		\begin{enumerate}
			\item There exists $i \in \{0,1,2\}$ such that $S = \omega^i \mathrm{Id}$.
			\item $\mathrm{tr}(S) = \mathrm{tr}(S^{-1}) = 0$.
		\end{enumerate}		    
   \end{enumerate}
 \end{lemme}
 \begin{proof} \hspace{1cm} 
  \begin{itemize}
  	\item[$(a) \Rightarrow (1)$]: Trivial
  	\item[$(b) \Rightarrow (1)$]: In this case, $\chi_S = X^3 - 1$. By Cayley-Hamilton theorem, we have $S^3 - \mathrm{Id} = 0$.
  	\item[$(1) \Rightarrow (2)$]: If $S^3 = \mathrm{Id}$, then $S$ is diagonalizable and its eigenvalues are cube roots of one. If $S$ has a triple eigenvalue, we are in case $(a)$. If not, since $\det(S) = 1$, the three eigenvalues are different and equal to $(1,\omega,\omega^2)$. We deduce that $\mathrm{tr}(S) = \mathrm{tr}(S^{-1}) = 1 + \omega + \omega^2 = 0 $.
  \end{itemize}

 \end{proof}
 
 We can now identify the irreducible components of
 $\mathcal{X}_{\sl3c}(\z3z3)$, thanks to the following proposition:
 
 \begin{prop}
  The algebraic set $\mathcal{X}_{\sl3c}(\z3z3)$ has 16 irreducible components : 15 isolated points and an irreducible component $X_0$ of complex dimension 4.
 \end{prop}
 
 \begin{proof}
  Consider $\mathcal{X}_{\sl3c}(\z3z3) \subset \mathcal{X}_{\sl3c}(F_2) \subset \mathbb{C}^9$, as the image of $\mathrm{Hom}(\z3z3 , \sl3c )$ by the trace maps of elements $s,t,st,st^{-1} , s^{-1},t^{-1},t^{-1}s^{-1},ts^{-1}$, and of the commutator $[s,t]$. Let $\rho \in \mathrm{Hom}(\z3z3 , \sl3c )$. Denote by $S = \rho(s)$ and $T = \rho(t)$. By Lemma \ref{lemme_ordre3_sl3c}, either $S$ or $T$ is a scalar matrix, or $\mathrm{tr}(S)=\mathrm{tr}(S^{-1})=\mathrm{tr}(T)=\mathrm{tr}(T^{-1})=0$. Let us deal with this two cases separately.\\
  
 \emph{First case: $S$ or $T$ is a scalar matrix.}
Suppose, for example, that $S = \omega^i\mathrm{Id}$ with $i \in \{0,1,2\}$. Since $T$ is of finite order and hence diagonalizable, the representation is totally reducible, and it is conjugate to either a representation of the form \[S=\omega^i\mathrm{Id} \hspace{2cm} T = \omega^j\mathrm{Id}\] with $i,j \in \{0,1,2\}$, either a representation given by \[S=\omega^i\mathrm{Id} \hspace{2cm} T = \begin{pmatrix}
\omega^2 & 0 &0 \\
0 & \omega &0 \\
0 & 0 &1
\end{pmatrix}
.\]
 
Considering the symmetries, we obtain 15 points of the character variety, classified by the traces of $S$ and $T$ in the following way (where $i,j \in \{0,1,2\}$):
\[
\begin{array}{|c||c|c|c|}
\hline
\mathrm{tr}(S) & 3\omega^i & 0 & 3\omega^i \\
\hline
\mathrm{tr}(T) & 3\omega^j & 3\omega^j & 0\\
\hline
\end{array}
 \]
 
    Since the traces of $S$ and $T$ are different for these 15 points and both $0$ in the second case, the points are isolated in $\mathcal{X}_{\sl3c}(\z3z3)$.
 
 \emph{Second case: $\mathrm{tr}(S)=\mathrm{tr}(S^{-1})=\mathrm{tr}(T)=\mathrm{tr}(T^{-1})=0$.}  
   By Lemma \ref{lemme_ordre3_sl3c}, all the points of $\mathcal{X}_{\sl3c}(F_2)$ satisfying this condition are in $\mathcal{X}_{\sl3c}(\z3z3)$. Denote by $z = \mathrm{tr}(ST)$, $z' = \mathrm{tr}((ST)^{-1})$, $w = \mathrm{tr}(ST^{-1})$, $w' = \mathrm{tr}(TS^{-1})$ and $x = \mathrm{tr}([S,T])$.
   The equation defining $\mathcal{X}_{\sl3c}(F_2) \subset \mathbb{C}^9$ becomes:
   
   \[x^2 - (zz'+ww' - 3)x + (zz'ww' + z^3 + z'^3 +w^3 + w'^3 - 6zz' - 6ww' +9) = 0\]
   
 This polynomial is irreducible. Indeed, if it were not, it would be equal to a product of two polynomials of degree 1 in $x$. By replacing $z'$, $w$ and $w'$ by $0$, we would obtain a factorization of the form $x^2 + 3x +z^3+9 = (x-R_1(z))(x-R_2(z))$, with $R_1(z)R_2(z) = z^3 + 9$ and $R_1(z)+R_2(z) = -3$. By considering the degrees of the polynomials $R_1$ and $R_2$ we easily obtain a contradiction.
 
   Since the polynomial defining $X_0$ is irreducible, $X_0$ is an irreducible component of $\mathcal{X}_{\sl3c}(\z3z3)$. Furthermore, it can be embedded into $\mathbb{C}^5$ and it is a ramified double cover of $\mathbb{C}^4$.
 \end{proof}
 
 \subsection{Reducible representations in the component $X_0 \subset \mathcal{X}_{\sl3c}(\z3z3)$}
 
 In order to complete the description of the character variety $\mathcal{X}_{\sl3c}(\z3z3)$, we  are going to identify the points corresponding to reducible representations. The 15 isolated points of the algebraic set come from totally reducible representations ; it remains to determine the points of the component $X_0$ corresponding to reducible representations.
 
\begin{notat} 
  We consider here $X_0 \subset \mathbb{C}^5$, with coordinates $(z,z',w,w',x)$ corresponding to the traces of the images of $(st,(st)^{-1}, st^{-1}, ts^{-1}, [s,t])$ respectively.
 We denote by $X_0^{\mathrm{red}}$ the image of reducible representations in $X_0$
\end{notat} 

\begin{rem}
 If the coordinates $(z,z',w,w',x)$ correspond to a reducible representation, then $\Delta(z,z',w,w') = 0$. Indeed, for a reducible representation, the two commutators $[s,t]$ and $[t,s]$ have the same trace, and the polynomial $X - Q(z,z',w,w')X + P(z,z',w,w')$ has a double root equal to those traces.
\end{rem}
 
 We are going to show that the locus of the characters of reducible representations is a set of 9 complex lines, which intersect at six points with triple intersections, corresponding to totally reducible representations. Before doing the proof let us fix a notation for these lines. 
 
 \begin{notat}
 For $i,j \in \{0,1,2\}$, let \[L^{(i,j)} = \{(z,z',w,w',x) \in X_0 \mid \omega^i z = \omega^{-i} z' ; \omega^j w = \omega^{-j} w' ; \omega^i z +  \omega^j w = 3\}. \]
 
 Each $L^{(i,j)}$ is a complex line parametrized by the coordinate $z$ (or $w$), and these lines intersect with triple intersections at the six points of coordinates $(z,w) = (0,3\omega^j)$ and $(z,w) = (3\omega^i,0)$, where $i,j \in \{0,1,2\}$.
 \end{notat}
 
With this notation, we can state in a simpler way the proposition describing the points of $X_0$ corresponding to reducible representations. 
 
 \begin{prop}\label{prop_rep_red_z3z3}
 The points of $X_0$ corresponding to reducible representations are exactly those in the lines $L^{(i,j)}$. In other terms, we have
 \[X_0^{\mathrm{red}} = \bigcup_{i,j \in \{0,1,2\}} L^{(i,j)}.\]
 \end{prop}

\begin{proof}
 We are going to show a double inclusion. Let us first show that \[\displaystyle X_0^{\mathrm{red}} \subset \bigcup_{i,j \in \{0,1,2\}} L^{(i,j)}.\]
 
Let $\rho \in \mathrm{Hom}(\z3z3 , \sl3c)$ be a reducible representation such that $\chi_\rho \in X_0$. Let $S = \rho(s)$ and $T = \rho(t)$. Since the representation is reducible, we can suppose, after conjugating $\rho$, that 
\[ 
S = \omega^i \begin{pmatrix}
S' & \\
 & 1
\end{pmatrix}
\text{ and }
T = \omega^j \begin{pmatrix}
T' & \\
 & 1
\end{pmatrix}
\]
where $i,j \in \{0,1,2\}$, and $S', T' \in \mathrm{SL}_2(\mathbb{C})$ of order 3 (and trace $-1$). Notice that it is enough to show that, whenever $i=j=0$, we have $\chi_\rho \in L^{(0,0)}$, in order to have the other cases by symmetry.
 Let us consider this case, with $i=j=0$. Since $S'T' \in \mathrm{SL}_2(\mathbb{C})$, we have $\mathrm{tr}(S'T') = \mathrm{tr}((S'T')^{-1})$, hence $\mathrm{tr}(ST) = \mathrm{tr}((ST)^{-1})$ and $z=z'$. Similarly, we know that $w=w'$. Furthermore, the trace equation in $\mathrm{SL}_2(\mathbb{C})$ gives
$\mathrm{tr}(S')\mathrm{tr}(T') = \mathrm{tr}(S'T') + \mathrm{tr}(S'T'^{-1})$.
Hence $(-1)^2 = (z-1) + (w-1)$, i.e. $z+w=3$. We obtain finally that $\chi_\rho \in L^{(0,0)}$.

Let us show now the other inclusion. In order to do it, it is enough to show that all the points of $L^{(0,0)}$ are images of representations given by 
\[ 
S =  \begin{pmatrix}
S' & \\
 & 1
\end{pmatrix}
\text{ and }
T = \begin{pmatrix}
T' & \\
 & 1
\end{pmatrix}
\]
 with $S',T' \in \mathrm{SL}_2(\mathbb{C})$, and recover the points of the other lines $L^{(i,j)}$ by considering $(\omega^i S, \omega^j T)$.
Since the image of every reducible representation of this form satisfies $z=z'$, $w=w'$ and $z+w=3$, we have to show that any $z\in \mathbb{C}$ can be written as $1 + \mathrm{tr}(S'T')$ with $S',T' \in \mathrm{SL}_2(\mathbb{C})$ of order 3 and trace $-1$. Fix $z \in \mathbb{C}$. Since the $\mathrm{SL}_2(\mathbb{C})$-character variety of $F_2$ is isomorphic to $\mathbb{C}^3$ by the trace maps of two generators and their product, there exist matrices $S', T' \in \mathrm{SL}_2(\mathbb{C})$ such that $(\mathrm{tr}(S'),\mathrm{tr}(T'),\mathrm{tr}(S'T')) = (-1,-1,z-1)$. In this case, the two matrices $S'$ and $T'$ have trace $-1$ and hence order 3, and we have $z = 1 + \mathrm{tr}(S'T')$.
\end{proof}

\begin{rem}
 The lines $L^{(i,j)}$ intersect with triple intersections at the six points of coordinates $(z,w) = (3\omega^i,0)$ and $(z,w) = (0,3\omega^i)$ with $i \in \{0,1,2\}$. The corresponding representations are exactly the totally reducible ones, where $S$ and $T$ are diagonal with eigenvalues $(1,\omega,\omega^2)$.
\end{rem} 
 
 \subsection{The fixed points of the involution $\Phi_2$}
 We are going to describe here the character varieties $\mathcal{X}_{\su21}(\z3z3)$ and  $\mathcal{X}_{\mathrm{SU}(3)}(\z3z3)$ as fixed points of the involution $\Phi_2$ of $\mathcal{X}_{\sl3c}(\z3z3)$. In this technical subsection, we will choose coordinates and find equations that describe the fixed points of $\Phi_2$. We will identify the characters corresponding to reducible representations as lying in an arrangement of 9 lines, and show that the ones corresponding to irreducible representations are in a smooth manifold of real dimension 4. We will describe the set obtained in this way in Subsection \ref{sous_sect_descr_chi_z3z3}.
 
 \begin{rem}
 Notice first that the 15 isolated points of $\mathcal{X}_{\sl3c}(\z3z3)$ come from totally reducible representations taking values in $\su21$ and $\mathrm{SU}(3)$. Hence they are in $\mathcal{X}_{\su21}(\z3z3) \cap \mathcal{X}_{\mathrm{SU}(3)}(\z3z3)$, and so in $\mathrm{Fix}(\Phi_2)$.
 \end{rem} 
 
 From now on, we will only consider the points of $\mathrm{Fix}(\Phi_2) \cap X_0$. Recall that we have identified $X_0$ to $\{(z,z',w,w',x) \in \mathbb{C}^5 \mid x^2 -Q(z,z',w,w')x + P(z,z',w,w') = 0\}$ by considering the trace maps of $st, (st)^{-1} , st^{-1} , ts^{-1}$ and $[s,t]$.
 
 \begin{rem}\label{rem_fix_x0}
  If $(z,z',w,w',x) \in  \mathrm{Fix}(\Phi_2) \cap X_0$, then $z' = \con{z}$ and $w' = \con{w}$. In this case, the polynomials $P$ and $Q$, that we will denote by $P(z,w)$ and $Q(z,w)$, take real values. Furthermore, we can write the discriminant of $X^2 -Q(z,w)X + P(z,w)$ as \[\Delta(z,w) = f(z) + f(w) - 2|z|^2|w|^2 +27,\]

where $f(z) = |z|^4 - 8\Re(z^3) + 18|z|^2 -27$ is the function described by Goldman in \cite{goldman} which is nonzero in the traces of regular elements of $\su21$ (positive for loxodromic elements, negative for elliptic elements). In a point of $\mathrm{Fix}(\Phi_2) \cap X_0$, the two roots of $X^2 -Q(z,w)X + P(z,w)$ are the traces of the images of the commutators $[s,t]$ and $[t,s]$. Since these commutators are inverses, and since we are in $\mathrm{Fix}(\Phi_2)$, the two roots are complex conjugate, which is equivalent to $\Delta(z,w) \leq 0$. 
 \end{rem}
 
 \begin{prop} We have: 
  \[ \mathrm{Fix}(\Phi_2) \cap X_0 = \{ (z,z',w,w',x) \in X_0 \mid z' = \con{z} , w' = \con{w} , \Delta(z,w) \leq 0 \}\]
 \end{prop}
 
 \begin{proof}
 We are going to show a double inclusion. The first one is given by Remark \ref{rem_fix_x0} ; Let us show the second.
 
 Let $z,w \in \mathbb{C}$ such that $\Delta(z,w) \leq 0$. Let $x$ be a root of $X^2 - Q(z,w)X + P(z,w)$. Since $\Delta(z,w) \leq 0$, the other root of the polynomial is $\con{x}$. We know that $(z,\con{z},w,\con{w},x) \in X_0$ ; we want to show that $(z,\con{z},w,\con{w},x) \in \mathrm{Fix}(\Phi_2)$. Let $\rho \in \mathrm{Hom}(\Gamma, \sl3c)$ be a semi-simple  representation with image in $X_0$ equal to $(z,\con{z},w,\con{w},x)$. It is enough to prove that for all $\gamma \in \Gamma$ we have $\mathrm{tr}(\rho(\gamma)) = \con{\mathrm{tr}(\rho(\gamma)^{-1})}$. But the representation ${}^t\!\rho^{-1}$ has image $(\con{z},z,\con{w},w,\con{x})$ in $X_0$. We deduce that the representations $\rho$ and ${}^t\!\con{\rho}^{-1}$ are semi-simple and have the same character. Hence they are conjugate and for all $\gamma \in \Gamma$ we have $\mathrm{tr}(\rho(\gamma)) = \con{\mathrm{tr}(\rho(\gamma)^{-1})}$, and so $(z,\con{z},w,\con{w},x) \in \mathrm{Fix}(\Phi_2)$.
 \end{proof}
 
 \begin{notat}
  From now on, we will consider $\mathrm{Fix}(\Phi_2) \cap X_0$ as $\{(z,w,x) \in \mathbb{C}^3 \mid \Delta(z,w) \leq 0 , x^2 - Q(z,w)x + P(z,w) =0\}$. The projection on the first two coordinates is a double cover of $\{(z,w) \in \mathbb{C}^2 \mid \Delta(z,w) \leq 0 \}$ outside from the level set $\Delta(z,w) = 0$, where points have a unique pre-image.
 \end{notat}
 
 We are going to identify the points corresponding to reducible representations, and then show that outside from these points, the $\su21$ and $\mathrm{SU}(3)$-character varieties are smooth manifolds.
 Let us begin by identifying the points of $X_0 \cap \mathrm{Fix}(\Phi_2)$ corresponding to reducible representations with the coordinates $(z,w)$.
 
 \begin{rem}
  Let $(z,w,x) \in X_0 \cap \mathrm{Fix}(\Phi_2)$. Let $\rho \in \mathrm{Hom}(\z3z3 , \sl3c)$ have coordinates $(z,\con{z},w,\con{w},x)$. The following assertions are equivalent:
  \begin{enumerate}
   \item $\rho$ is reducible.
   \item There exist $i,j \in \{0,1,2\}$ such that $\omega^i z$ and $\omega^jw$ are real and $\omega^i z + \omega^jw = 3$.
  \end{enumerate}
  In this case $\Delta(z,w) = 0$.
 \end{rem}
 \begin{proof}
   It is an immediate consequence of Proposition \ref{prop_rep_red_z3z3} and the fact that we are in $\mathrm{Fix}(\Phi_2)$ and hence, in the coordinates $(z,z',w,w',x)\in \mathbb{C}^5$, we have $z'=\con{z}$ and $w' = \con{w}$. At last, we check that $\Delta(z,3-z) = 0$. If we are in the setting of the equivalence, we have $\Delta(z,w) = 0$.
\end{proof}
 
 Let us show now that the points corresponding to irreducible representations form a smooth manifold.
 
 \begin{prop}\label{prop_rep_irr_lisses}
 Outside from the points corresponding to reducible representations, the set $X_0 \cap \mathrm{Fix}(\Phi_2)$ is a sub-manifold of $\mathbb{C}^3$ of real dimension 4.
 \end{prop}
 
 \begin{proof}
 Recall that we defined $X_0 \cap \mathrm{Fix}(\Phi_2)$ as:

\[\{ (z,w,x) \in \mathbb{C}^3 \mid x^2 -Q(z,w)x+P(z,w) =0 , \Delta(z,w) \leq 0 \} \]
where 
\[ Q(z,w) =|z|^2 + |w|^2 - 3 \]
\[ P(z,w) = 2\Re(z^3) + 2\Re(w^3) + |z|^2|w|^2 - 6|z|^2 -6|w|^2 + 9.\]

We can hence re-write
$X_0 \cap \mathrm{Fix}(\Phi_2)$ as :

\[\{ (z,w,x) \in \mathbb{C}^3 \mid x + \con{x} = Q(z,w) , x\con{x} = P(z,w) \} . \]

Consider the functions $f_1 , f_2 : \mathbb{C}^3 \rightarrow \mathbb{R}$ given by 
$f_1(z,w,x) = Q(z,w) - ( x + \con{x})$ and $f_2(z,w,x) = P(z,w) - (x  \con{x})$, and then $f = (f_1,f_2) : \mathbb{C}^3 \rightarrow \mathbb{R}^2$. With this notation, $X_0 \cap \mathrm{Fix}(\Phi_2) = f^{-1}(\{0\})$. We are going to show that outside from the points corresponding to reducible representations, $f$ is a submersion, i.e. that $\mathrm{d}f$ is of rank 2.

 Let $(z_0,w_0,x_0) \in X_0 \cap \mathrm{Fix}(\Phi_2)$. 
Notice first that \[\frac{\partial f}{\partial x}(z_0,w_0,x_0) = \begin{pmatrix}
  -1 \\ -\con{x_0}
  \end{pmatrix} \text{ and } \frac{\partial f}{\partial  \con{x}} = \begin{pmatrix}
  -1 \\ -x_0
  \end{pmatrix},\]
   hence $\mathbb{d}f$ is always of rank at least $1$ and, if $x_0 \notin \mathbb{R}$, the map $f$ is a submersion at $(z_0,w_0,x_0)$.
  Suppose now that $\mathrm{d}f (z_0,w_0,x_0)$ is of rank $1$. In particular, $x_0 \in \mathbb{R}$. We want to show that in this case, the point $(z_0,w_0,x_0)$ corresponds to a reducible representation.
  
  We have \[z_0\frac{\partial f}{\partial z}(z_0,w_0,x_0) = \begin{pmatrix}
  |z_0|^2 \\ 3z_0^3 + |z_0|^2|w_0|^2 - 6|z_0|^2
  \end{pmatrix}
   \text{ and }\]
    \[\con{z_0}\frac{\partial f}{\partial  \con{z}}(z_0,w_0,x_0) = \begin{pmatrix}
  |z_0|^2 \\ 3\con{z_0}^3 + |z_0|^2|w_0|^2 - 6|z_0|^2
  \end{pmatrix},\]
  
  and, since the two vectors linearly dependent, we have $z_0^3 \in \mathbb{R}$. In the same way, $w_0^3 \in \mathbb{R}$. Then there exist $r_1,r_2 \in \mathbb{R}$ and $i,j \in \{0,1,2\}$ such that $z_0 = \omega^i r_1$ and $w_0 = \omega^j r_2 $.
  By Proposition \ref{prop_rep_red_z3z3}, it is enough to show that $r_1 + r_2 = 3$ in order to finish the proof. We consider two cases:

\emph{First case : $r_1$ or $r_2$ is zero.}
 Suppose, for example, that $r_2 = 0$. In this case, since $f(0,0,x_0) \neq (0,0)$, we have $r_1 \neq 0$. On the one hand, we have
 $2x_0 = Q(z,0) = r_1^2 - 3$. On the other hand, since $z_0\frac{\partial f}{\partial z}(z_0,w_0,x_0)$ and $\frac{\partial f}{\partial x}(z_0,w_0,x_0)$ are linearly dependent, we have $x_0 = 3(r_1-2)$. We deduce that $6(r_1-2) = 2x_0 = r_1^2 - 3$, hence $r_1^2 - 6r_1 + 9 = 0$ and $r_1 = 3$.
 
\emph{Second case : $r_1 , r_2 \neq 0$.}  We know that the following vectors are collinear :
 \[z_0\frac{\partial f}{\partial z}(z_0,w_0,x_0) =
 \begin{pmatrix}
  |z_0|^2 \\ 3z_0^3 + |z_0|^2|w_0|^2 - 6|z_0|^2 \end{pmatrix} = r_1^2 \begin{pmatrix}
  1 \\ 3r_1 + r_2^2 -6 \end{pmatrix} \]
  \[w_0\frac{\partial f}{\partial w}(z_0,w_0,x_0) =\begin{pmatrix}
  |w_0|^2 \\ 3w_0^3 + |w_0|^2|z_0|^2 - 6|w_0|^2 \end{pmatrix} = r_2^2 \begin{pmatrix}
  1 \\ 3r_2 + r_1^2 -6 \end{pmatrix}.\]
 We deduce that $3r_1 + r_2^2 -6 = 3r_2 + r_1^2 -6$. If $r_1 \neq r_2$, then they are the two roots of a polynomial of the form $X^2 - 3X +k$, hence $r_1 + r_2 =3$. If not, and $r_1 = r_2$ we have $2x_0 = Q(z_0,w_0) = 2r_1^2 - 6$ and, since $z_0\frac{\partial f}{\partial z}(z_0,w_0,x_0)$ and $\frac{\partial f}{\partial x}(z_0,w_0,x_0)$ are collinear, $x_0 = r_1^2 - 3r_1 -6$. We deduce that $r_1 = \frac{3}{2}$ and $r_1 + r_2 = 3$.
 \end{proof}

 \subsection{Description of $\mathcal{X}_{\su21}(\z3z3)$ and $\mathcal{X}_{\mathrm{SU}(3)}(\z3z3)$} \label{sous_sect_descr_chi_z3z3}
 
 We are going to describe here the character varieties $\mathcal{X}_{\su21}(\z3z3)$ and  $\mathcal{X}_{\mathrm{SU}(3)}(\z3z3)$. In order to do it, we are going to study in detail $\mathrm{Fix}(\Phi_2)$, verify that it is the union of the two character varieties, and that their intersection corresponds to reducible representations. We finally consider two slices of $\mathrm{Fix}(\Phi_2)$, that were studied respectively by Parker and Will in \cite{ParkerWill} and by Falbel, Guilloux, Koseleff, Rouiller and Thistlethwaite in \cite{character_sl3c}.
 
 First, consider the $15$ isolated points of $\mathcal{X}_{\sl3c}(\z3z3)$, which are all in $\mathrm{Fix}(\Phi_2)$. They correspond to totally reducible representations. Since an order 3 matrix is conjugated to a matrix in $\su21$ and $\mathrm{SU}(3)$, we have the following remark:
 
 \begin{rem}
The points of $\mathrm{Fix}(\Phi_2)$ corresponding to totally reducible representations are all in $\mathcal{X}_{\su21}(\z3z3) \cap \mathcal{X}_{\mathrm{SU}(3)}(\z3z3)$.
 \end{rem}
 
 It remains to consider the representations of $X_0 \cap \mathrm{Fix}(\Phi_2)$. Proposition \ref{prop_inter_charvar_reelle} ensures us that points corresponding to irreducible representations are exactly in one of the character varieties $\mathcal{X}_{\su21}(\z3z3)$ and $\mathcal{X}_{\mathrm{SU}(3)}(\z3z3)$. For the points of $X_0$ corresponding to reducible representations, we briefly modify the proof of Proposition \ref{prop_rep_red_z3z3} in order to obtain the following remark:
 
 \begin{rem}\label{rem_tot_red_z3z3}
  The points of $\mathrm{Fix}(\Phi_2) \cap X_0$ corresponding to reducible representations are in $\mathcal{X}_{\su21}(\z3z3)$. Only some of them are in  $\mathcal{X}_{\mathrm{SU}(3)}(\z3z3)$.
 \end{rem}
\begin{proof}
 A reducible representation $\rho$ with character in $X_0$ is conjugate to a representation given by
 \[ 
S = \omega^i \begin{pmatrix}
S' & \\
 & 1
\end{pmatrix}
\text{ and }
T = \omega^j \begin{pmatrix}
T' & \\
 & 1
\end{pmatrix}
\]
with $i,j \in \{0,1,2\}$, and $S', T' \in \mathrm{SL}_2(\mathbb{C})$ of order 3 (and trace $-1$). Since $\chi_\rho \in \mathrm{Fix}(\Phi_2)$, the representation $\rho' : \z3z3 \rightarrow \mathrm{SL}_2(\mathbb{C})$ given by $\rho' (s) = S'$ and $\rho' (t) = T'$, is in $\mathrm{Fix}(\Phi_2) \subset \mathcal{X}_{\mathrm{SL}_2(\mathbb{C})}(\z3z3)$. If $\rho$ is totally reducible, then, by Remark \ref{rem_tot_red_z3z3}, $\chi_\rho \in \mathcal{X}_{\su21}(\z3z3) \cap \mathcal{X}_{\mathrm{SU}(3)}(\z3z3)$. If not, $\rho'$ is irreducible and, by Proposition \ref{prop_invol_traces}, maybe after a conjugation of $\rho'$, we have $S',T' \in \mathrm{SU}(2)$ or $\mathrm{SU}(1,1)$. If $S',T' \in \mathrm{SU}(2)$, then $S,T \in \su21 \cap \mathrm{SU}(3)$. If, on the other hand, $S',T' \in \mathrm{SU}(1,1)$, then $S,T \in \su21$. It remains to see that the second case happens.  The point $(4,-1,7) \in X_0 \cap \mathrm{Fix}(\Phi_2)$ corresponds to a reducible representation, with a trace $4$ element, and so it cannot take values in $\mathrm{SU}(3)$.
\end{proof} 
 
 Furthermore, by noticing that an irreducible representation cannot take values at the same time in $\su21$ and in $\mathrm{SU}(3)$, we obtain the following proposition :
 
 \begin{prop}We have
  \[\mathrm{Fix}(\Phi_2) = \mathcal{X}_{\su21}(\z3z3) \cup \mathcal{X}_{\mathrm{SU}(3)}(\z3z3).\]
   The subsets $ \mathcal{X}_{\su21}(\z3z3)$ and $ \mathcal{X}_{\su21}(\z3z3)$ are non-empty and intersect only at points corresponding to reducible representations.
 \end{prop}
 
At last, we are going to draw some slices of $\mathrm{Fix}(\Phi_2)$, corresponding to projections on coordinates $(z,w)$, followed by a restriction to a slice of the form $z=z_0$ or $w = w_0$.
Recall that the projection on coordinates $(z,w)$ is a double cover outside from the level set $\Delta(z,w) = 0$, where points have a unique pre-image.
 We draw, in a plane of the form $(z,w_0)$, the curve $\Delta(z,w_0)=0$, and then we identify the regions contained in $\mathcal{X}_{\su21}(\z3z3)$ and those contained in $\mathcal{X}_{\mathrm{SU}(3)}(\z3z3)$.
 
 \subsubsection{The Parker-Will slice}
  In their article \cite{ParkerWill}, Parker and Will give an explicit parametrization of representations of $\z3z3 = \langle s,t \rangle$ taking values in $\su21$ such that the image of $st$ is unipotent. This corresponds exactly to representations such that the trace of the image of $st$ is equal to $3$. They form a family of representations of the fundamental group of the Whitehead link complement containing the holonomy representation of a \CR {} uniformization of the manifold. This particular representation has coordinates $(z,w,x) = (3,3,\frac{15 + 3i\sqrt{15}}{2})$. We can see this slice in figure \ref{fig_tranche_parker_will}. We see three lobes corresponding to representations taking values in $\su21$, which intersect in a singular point, of coordinate $z=0$, which corresponds to a totally reducible representation, of coordinates $(z,w,x)= (0,3,3)$. Going back to coordinates $(z,w,x)$ on $X_0 \cap \mathrm{Fix}(\Phi_2)$, the representations of the slice $z = 0$ form, topologically, three spheres touching at a single point.

 \begin{figure}[ht]
 \center
  \includegraphics[width=6.5cm]{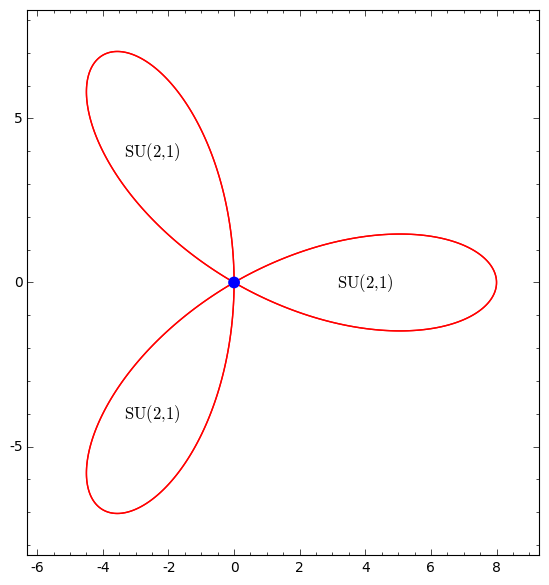}
  \caption{The Parker-Will slice of $\mathcal{X}_{\su21}(\z3z3)$.} \label{fig_tranche_parker_will}
 \end{figure}

 \subsubsection{The Thistlethwaite slice}
  In the last section of their article \cite{character_sl3c}, Falbel, Guilloux, Koseleff, Roullier and Thistlethwaite give an explicit parametrization of representations lifting the irreducible components $R_1$, $R_2$ and $R_3$ of $\mathcal{X}_{\sl3c}(\Gamma_8)$, as we saw in Subsection \ref{sous_sect_char_sl3c_m8}. They also give necessary and sufficient conditions for a representation to take values in $\su21$ or $\mathrm{SU}(3)$ : therefore they parametrize lifts of the intersections of $R_1$ and $R_2$ with $ \mathcal{X}_{\su21}(\Gamma_8)$ and $\mathcal{X}_{\mathrm{SU}(3)}(\Gamma_8)$. 
Recall that the fundamental group of the figure eight knot complement has the following presentation:

\[\Gamma_8 = \langle g_1, g_2 , g_3 \mid g_2 = [g_3 , g_1^{-1}] , g_1 g_2 = g_2 g_3 \rangle \]  
  As noticed by Deraux in \cite{deraux_uniformizations} and by Parker and Will in \cite{ParkerWill}, if $G_1$, $G_2$ and $G_3$ are the images of $g_1,g_2$ and $g_3$ respectively by a representation with character in $R_2$, then $(G_1G_2)=(G_1^2G_2)^3 = G_2^4 = \mathrm{Id}$. Setting $T = (G_1G_2)^{-1}$ and $S = (G_1^2G_2)$, we have two elements of $\sl3c$ of order 3 which generate the image of the representation, since $G_1 = ST , G_3 = TS$ and $G_2 = (TST)^{-1} = (TST)^{3}$. Hence we can consider $R_2 \subset \mathcal{X}_{\sl3c}(\z3z3)$ : this component corresponds to the slice of coordinate $w = 1$, since $TST$ has order 4 if and only if $\mathrm{tr}(TST) = \mathrm{tr}(ST^{2}) = \mathrm{tr}(ST^{-1}) = 1$. We can see this slice in figure \ref{fig_tranche_fgkrt}. It has three regions of representations taking values in $\su21$ and a region of representations taking values in $\mathrm{SU}(3)$. They intersect at three singular points, corresponding to reducible representations.
Going back to coordinates $(z,x)$ on on the slice $w=1$ of $X_0 \cap \mathrm{Fix}(\Phi_2)$, these regions are the images of four topological spheres which intersect at three points.
 
  \begin{figure}[htbp]
 \center
  \includegraphics[width=7cm]{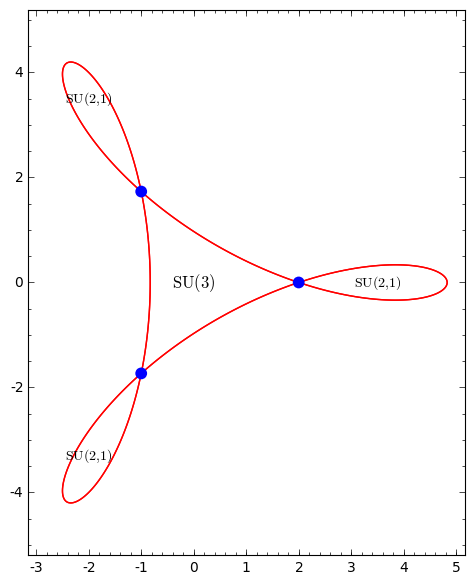}
  \caption{The slice of Falbel, Guilloux, Koseleff, Roullier and Thistlethwaite of $\mathcal{X}_{\su21}(\z3z3) \cup \mathcal{X}_{\mathrm{SU}(3)}(\z3z3)$. It is given by $w = 1$} \label{fig_tranche_fgkrt}
 \end{figure}
 
  \subsubsection{Other remarkable slices}
 At last, to complete the whole picture, we describe three more slices of $X_0 \cap \mathrm{Fix}(\Phi_2)$. Recall that, thanks to Proposition \ref{prop_rep_irr_lisses}, a slice of the form $w=w_0$ will only have singular points if $w_0^3 \in \mathbb{R}$. On the one hand, in figure \ref{fig_tranches_z3z3_legende}, we see the slices $w = 3.5$ and $w=3.5 + 0,1 i$. In each one there are three regions corresponding to irreducible representations taking values in $\su21$, which intersect, in the slice $w = 3.5$ at three points corresponding to reducible representations. There are no points corresponding to representations with values in $\mathrm{SU}(3)$.
 
 \begin{figure}[htbp]
\center
\begin{subfigure}{0.4\textwidth}
 \includegraphics[width=4.5cm]{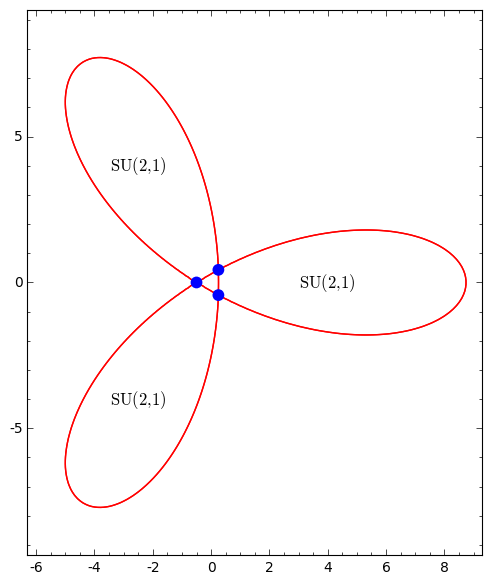}
 \caption{The slice $w_0= 3.5$. There are three singular points.}
 \end{subfigure} \hspace{1cm}
\begin{subfigure}{0.4\textwidth}
 \includegraphics[width=4.5cm]{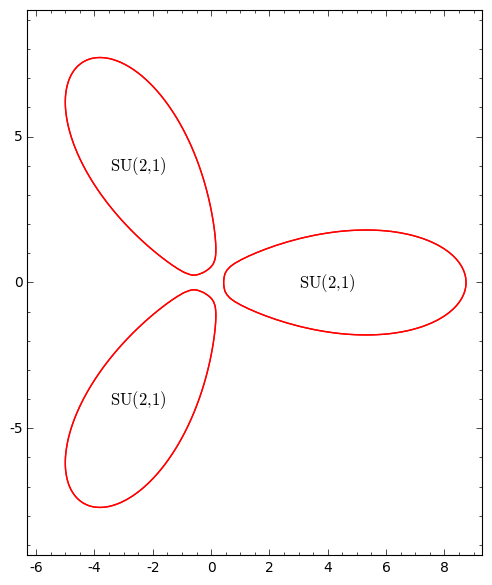}
 \caption{The slice $w_0 = 3.5 + 0,1i$. The region is smooth.}
 \end{subfigure}
 \caption{The slices $w_0= 3.5$ and $w_0 = 3.5 + 0.1i$.} \label{fig_tranches_z3z3_legende}
\end{figure} 
 
 On the other hand, in figure \ref{fig_tranche_z3z3_1_legende}, we see the slice $w= 1 + 0,1i$ : there are three regions corresponding to irreducible representations taking values in $\su21$ and a region corresponding to representations taking values in $\mathrm{SU}(3)$.
  
  \begin{figure}[htbp]
 \center
  \includegraphics[width=7cm]{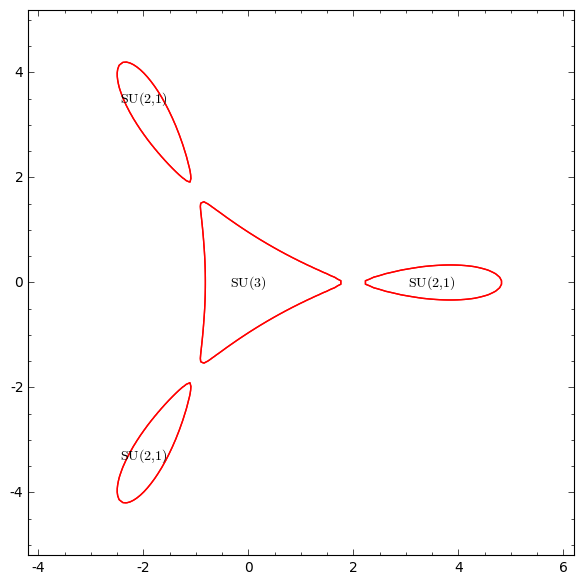}
  \caption{The slice $w_0 = 1 + 0,1i$. The region is smooth.} \label{fig_tranche_z3z3_1_legende}
 \end{figure}

\break

 \bibliographystyle{alpha}
\bibliography{./tex/biblio_these_updated}

\newcommand{\etalchar}[1]{$^{#1}$}
\begin{thebibliography}{CCG{\etalchar{+}}94}

\bibitem[CCG{\etalchar{+}}94]{cooper_plane_1994}
D.~Cooper, M.~Culler, H.~Gillet, D.~D. Long, and P.~B. Shalen.
\newblock Plane curves associated to character varieties of 3-manifolds.
\newblock 118(1):47--84, 1994.

\bibitem[CL98]{cooper_representation_1998}
D.~Cooper and D.~D. Long.
\newblock Representation theory and the {{A}}-polynomial of a knot.
\newblock {\em Chaos, Solitons and Fractals}, 9(4-5):749--763, 1998.
\newblock Knot theory and its applications.

\bibitem[CS83]{culler_shalen}
Marc Culler and Peter~B. Shalen.
\newblock Varieties of {{Group Representations}} and {{Splittings}} of
  3-{{Manifolds}}.
\newblock {\em The Annals of Mathematics}, 117(1):109, January 1983.

\bibitem[Der14]{deraux_uniformizations}
Martin Deraux.
\newblock A 1-parameter family of spherical {{CR}} uniformizations of the
  figure eight knot complement.
\newblock {\em arXiv preprint arXiv:1410.1198}, 2014.

\bibitem[DF15]{falbel}
Martin Deraux and Elisha Falbel.
\newblock Complex hyperbolic geometry of the figure eight knot.
\newblock {\em Geometry \& Topology}, 19(1):237--293, February 2015.

\bibitem[Fal08]{falbel_spherical}
Elisha Falbel.
\newblock A spherical {{CR}} structure on the complement of the figure eight
  knot with discrete holonomy.
\newblock {\em Journal of Differential Geometry}, 79(1):69--110, 2008.

\bibitem[FGK{\etalchar{+}}16]{character_sl3c}
E.~Falbel, A.~Guilloux, P.-V. Koseleff, F.~Rouillier, and M.~Thistlethwaite.
\newblock Character varieties for {{SL}}(3,{{C}}): The figure eight knot.
\newblock {\em Experimental Mathematics}, 25(2):219--235, 2016.

\bibitem[FL09]{florentino_topology_2009}
Carlos Florentino and Sean Lawton.
\newblock The topology of moduli spaces of free group representations.
\newblock {\em Mathematische Annalen}, 345(2):453--489, October 2009.

\bibitem[Gol88]{goldman_topological_1988}
William~M. Goldman.
\newblock Topological components of spaces of representations.
\newblock {\em Inventiones mathematicae}, 93(3):557--607, October 1988.

\bibitem[Gol99]{goldman}
William~M. Goldman.
\newblock {\em Complex Hyperbolic Geometry}.
\newblock Oxford Mathematical Monographs. {The Clarendon Press, Oxford
  University Press, New York}, 1999.

\bibitem[Gol04a]{goldman_complex-symplectic_2004}
William~M. Goldman.
\newblock The complex-symplectic geometry of {{SL}}(2,{{C}})-characters over
  surfaces.
\newblock In {\em Algebraic Groups and Arithmetic}, pages 375--407. {Tata Inst.
  Fund. Res., Mumbai}, 2004.

\bibitem[Gol04b]{goldman_exposition_2004}
William~M. Goldman.
\newblock An exposition of results of {{Fricke}}.
\newblock {\em arXiv preprint math/0402103}, 2004.

\bibitem[Hel08]{helgason_geometric_2008}
Sigurdur Helgason.
\newblock {\em Geometric Analysis on Symmetric Spaces}, volume~39 of {\em
  Mathematical Surveys and Monographs}.
\newblock {American Mathematical Society, Providence, RI}, second edition,
  2008.

\bibitem[Heu16]{heusener_slnc_2016}
Michael Heusener.
\newblock {{SL}}(n,{{C}})- {{Representation}} spaces of {{Knot Groups}}.
\newblock {\em arXiv:1602.03825 [math]}, February 2016.

\bibitem[HMP15]{heusener_sl3c-character_2015}
Michael Heusener, Vicente Munoz, and Joan Porti.
\newblock The {{SL}}(3,{{C}})-character variety of the figure eight knot.
\newblock {\em arXiv:1505.04451 [math]}, May 2015.

\bibitem[Law07]{lawton_generators_2007}
Sean Lawton.
\newblock Generators, relations and symmetries in pairs of unimodular matrices.
\newblock {\em Journal of Algebra}, 313(2):782--801, July 2007.

\bibitem[LM85]{lubotzky_varieties_1985}
Alexander Lubotzky and Andy~R. Magid.
\newblock Varieties of representations of finitely generated groups.
\newblock {\em Memoirs of the American Mathematical Society}, 58(336):0--0,
  1985.

\bibitem[MS84]{morgan_shalen}
John~W. Morgan and Peter~B. Shalen.
\newblock Valuations, {{Trees}}, and {{Degenerations}} of {{Hyperbolic
  Structures}}, {{I}}.
\newblock {\em Annals of Mathematics}, 120(3):401--476, 1984.

\bibitem[Par11]{parreau_espaces_2011}
Anne Parreau.
\newblock Espaces de repr{\'e}sentations completement r{\'e}ductibles.
\newblock {\em Journal of the London Mathematical Society}, 83(3):545--562,
  2011.

\bibitem[Par12]{parreau_compactification_2012}
Anne Parreau.
\newblock Compactification d'espaces de repr{\'e}sentations de groupes de type
  fini.
\newblock {\em Mathematische Zeitschrift}, 272(1-2):51--86, October 2012.

\bibitem[Pro76]{procesi_invariant_1976}
C~Procesi.
\newblock The invariant theory of n x n matrices.
\newblock {\em Advances in Mathematics}, 19(3):306--381, March 1976.

\bibitem[PS85]{procesi_inequalities_1985}
Claudio Procesi and Gerald Schwarz.
\newblock Inequalities defining orbit spaces.
\newblock {\em Inventiones mathematicae}, 81(3):539--554, October 1985.

\bibitem[PW15]{ParkerWill}
John~R. Parker and Pierre Will.
\newblock A complex hyperbolic {{Riley}} slice.
\newblock {\em arXiv preprint arXiv:1510.01505}, 2015.

\bibitem[Sha02]{shalen_representations_2002}
Peter~B. Shalen.
\newblock Representations of 3-manifold groups.
\newblock In {\em Handbook of Geometric Topology}, pages 955--1044.
  {North-Holland, Amsterdam}, 2002.

\bibitem[Sik12]{sikora_character_2012}
Adam Sikora.
\newblock Character varieties.
\newblock {\em Transactions of the American Mathematical Society},
  364(10):5173--5208, 2012.

\bibitem[Wil15]{will_generateurs}
Pierre Will.
\newblock Two {{Generator}} groups acting on the complex hyperbolic plane.
\newblock {\em arXiv preprint arXiv:1510.01500}, 2015.

\end{thebibliography}

\end{document}